\documentclass[11pt, a4paper]{amsart}

\usepackage[a4paper, margin=3cm]{geometry}
\usepackage{amssymb}
\usepackage{xspace}
\usepackage{nicefrac}
\usepackage{amssymb,amsfonts}
\usepackage{enumitem}
\usepackage{url}
\usepackage{latexsym}
\usepackage{color}
\usepackage{hyperref}
\hypersetup{colorlinks=true,citecolor=blue}

\newtheorem{Theorem}{Theorem}[section]
\newtheorem{Definition}[Theorem]{Definition}

\newtheorem{Lemma}[Theorem]{Lemma}
\newtheorem{Corollary}[Theorem]{Corollary}

\newtheorem{Remark}[Theorem]{Remark}
\newtheorem*{acknowledgements}{Acknowledgements}

\def\C{\mathbb C}

\def\N{\mathbb N}

\newcommand{\CA}{\mathcal{A}}
\newcommand{\CB}{\mathcal{B}}

\newcommand{\CD}{\mathcal{D}}

\newcommand{\CF}{\mathcal{F}}

\newcommand{\CH}{\mathcal{H}}

\newcommand{\CJ}{\mathcal{J}}
\newcommand{\CK}{\mathcal{K}}
\newcommand{\CL}{\mathcal{JL}}

\newcommand{\CP}{\mathcal{P}}

\newcommand{\CS}{\mathcal{S}}

\def\be{\begin{equation}}
\def\ee{\end{equation}}
\def\bt{\begin{Theorem}}
\def\et{\end{Theorem}}
\def\bi{\begin{itemize}}
\def\ei{\end{itemize}}
\def\bea{\begin{eqnarray}}
\def\eea{\end{eqnarray}}
\def\beast{\begin{eqnarray*}}
\def\eeast{\end{eqnarray*}}
\def\ben{\begin{enumerate}}
\def\een{\end{enumerate}}

\def\half{\frac{1}{2}}
  
\def\bi{\bibitem}

\def\rar{\rightarrow}

\def\e{{\epsilon}}

\begin{document}
\title{CAR flows on type $III$ factors and its extendability
}
\author{Panchugopal Bikram}
\address{Institue of mathematical sciences, Chennai, India}
\email{pg.math@gmail.com}


\thanks{ }

\begin{abstract}
In this paper 
using one of the 
necessary conditions obtained for extendability in \cite{BISS}, we prove that  
the CAR flows (\cite{G}) on type $III$
factors arising from most quasi-free states are not extendable.
As a consequence we find the super product system of CAR flows.
We know from \cite{Arv} that CCR flows and CAR flows on 
type $I$ factors with the same Arveson index are cocycle
conjugate. But our result  
together with \cite{BISS} will show that CCR flows and 
CAR flows on type $III$ factors are not cocycle conjugate. 
\end{abstract} 
\maketitle \tableofcontents

\medskip
\noindent

\section{Introduction}
A weak-* continuous semigroup of of unital $*$-endomorphisms on a von Neumann 
algebra is called an $E_0$-semigroup. $E_0$-semigroups on type $I$ factors
have receivded much atention (see the monograph \cite{Arv} for an
extensive reference). 
The study of $E_0$-semigroups on type $II_1$ factors was initiated by
Powers in 1998 (see \cite{Pr}). There was little progress on
$E_0$-semigroups on type $II_1$ factors until the results independently obtained recently in \cite{Alev} and  \cite{MS}).

On the other hand, $E_0$-semigroups  on type $III$ factors had not received too much attention. 

In 2001, G.G. Amosov, A.V Bulinski and Shirkov initiated a study of $E_0$-semigroups on 
arbitrary factors (see \cite{ABS}). They were interested in a special kind of $E_0$-semigroups 
which they called ``regular $E_0$-semigroups'', which are those that can be extended to type $I$ 
factors in a canonical 
way. Then in 2013, in \cite{BISS}, we studied a certain class of endomorphisms
and $E_0$-semigroups on arbitrary factors which we call extendable. 
Unfortunately  \cite{ABS} has some errors: for example they claim
in section 5 of their paper,
that CAR flows arising from quasi-free state of $\frac{1}{2}$ are regular semigroups. 
We proved in \cite{BISS}  that these CAR flows are not regular semigroups
(although we refer to them as ``extendable'' $E_0$-semigroups).

Section 2 is the preliminary section. Where we will briefly discuss the meaning of extendability 
of an $E_0$-semigroup. We also mention the definition of supper product system and then we write some results
regarding the supper product system.

In section 3 we study CAR flows on type $III$
factors and prove that they are not extendable. Also we point out an
error in \cite{ABS} and we find out the supper product system of CAR flows.

In section 4 we study the relations between CCR  and CAR flows on type
$III$ factors and then we distinguish them.

\section{Preliminaries}
\subsection{Extendable $E_0$-semigroup}
We begin by recalling some facts from \cite{BISS} that will be used often in the sequel. Assume that
$\phi$ is a faithful normal state on a factor $M$.
Let $\lambda_M$ be the left regular representation of $M$  onto $\CB(L^2(M,\phi))$.
Identify $x \in M$ with $\lambda_M(x)$. 
For the modular conjugation
operator, we simply write $\CJ$  for  ($\CJ_\phi$).
Thanks to the Tomita-Takesaki theorem, we know that
\begin{itemize}
 \item $\jmath := \CJ (\cdot)\CJ$ is a *-preserving conjugate-linear 
isomorphism of $\CB(L^2(M,\phi))$ onto itself, which maps $M$ and
$M^\prime$ onto one another, and that
 \item $\widehat{1_M}$ is a cyclic and separating vector for
$M^\prime$.
\end{itemize}

 We assume that $\theta$ is a normal unital 
*-endomorphism, which preserves $\phi$. The invariance assumption  
$\phi \circ \theta = \phi$ implies that there exists a unique isometry
$u_\theta$ on $L^2(M,\phi)$ such that $u_\theta x \widehat{1_M} = \theta(x)
\widehat{1_M}$, which in turn implies that $u_\theta x = \theta(x) u_\theta
~\forall x \in M$. Recall the following definition from \cite{BISS}.  
\begin{Definition}\label{equi-modular}
If $M,\phi, \theta$ are as above, and if the associated isometry $u_\theta$ of $L^2(M,\phi)$ commutes with 
the modular conjugation operator $\CJ (=\CJ_\phi)$, then
$\theta$ is called an {\bf equi-modular} endomorphism of the  
factorial non-commutative probability space $(M, \phi)$. 
\end{Definition}
Suppose $\theta$ is an equi-modular endomorphism of a factorial non-commutative probability space $(M, \phi)$.  Then 
the equation $\theta^\prime (x^\prime) = \CJ \theta(\CJ x^\prime \CJ )\CJ$
 defines a unital normal *-endomorphism of $M^\prime$, which preserves the state given by
$\phi^\prime(x^\prime) = \overline{\phi(\CJ x^\prime  \CJ)}$; and we have the identifications
$ L^2(M',\theta ') = L^2(M,\phi)$,
$\widehat{1_{M'}} = \widehat{1_M}$ and 
$u_{\theta '}  = u_\theta $ (for details see \cite{BISS}).
For convenience of reference, we include this definition from \cite{BISS} .

\begin{Definition}
Let $\theta$ be an equi-modular endomorphism of a factorial
non-commutative probability space $(M, \phi)$ in standard
form (i.e., viewed as embedded in $\CB(L^2(M,\phi))$ as above).
Then $\theta$ is called {\bf extendable } if 
 there exists a unital normal $*$-endomorphism $\theta^{(2)}$ of
  $\CB(L^2(M,\phi))$ such that $\theta^{(2)}(x) = \theta(x)$ and
  $\theta^{(2)}(j(x)) = j(\theta(x))$ for all $x \in M$. 

\end{Definition}

\begin{Definition}
$\{\alpha_t: t \geq 0\}$ is said to be an $E_0$-semigroup on a von
Neumann probability space $(M,\phi)$, if
\ben
\item $\alpha_t$ is a $\phi$-preserving normal unital *-homomorphism
  of $M$ for each $t \geq 0$;
\item $\alpha_t \circ \alpha_s = \alpha_{t+s} ~\forall s.t 'geq 0$;
\item $\alpha_0 = id_M$; and
\item $[0,\infty) \ni t \mapsto \rho(\alpha_t(x))$ is continuous for
  each $x \in M, \rho \in M_*$.
\een
It is called extendable $E_0$-semigroup if fol all $t\geq 0$, $\alpha_t$ is 
extendable.
\end{Definition}

\subsection{Super product system}

The notion of super product system is already introduced in 
\cite{MS}. It is a generalization of the
product systems introduced by Arveson, and may help to 
analyse $E_0$-semigroups on non-type $I$ factors (\cite{MS}). 
\begin{Definition}
 A super product system of Hilbert spaces is a one
parameter family of separable Hilbert spaces $\{H_t : t \geq 0 \}$, together
with isometries
\[
 U_{s,t} : H_s \otimes H_t \mapsto H_{s+t},
\]
for $s, t \in (0, \infty)$,
which satisfy the 
following requirements of associativity and measurability:\\
$(i)$ (Associativity) For any $s_1 , s_2 , s_3 \in (0, \infty)$
\[
U_{s_1 ,s_2 + s_3} (1_{H_{s_1}} \otimes U_{s_2 , s_3 }) = U_{s_1 +s_2 ,s_3}  (U_{s_1 ,s_2} \otimes 1_{H_{s_3}}).
\]
$(ii)$ (Measurability) The space $H = \{(t, \xi_t ) : t ∈ (0, \infty), \xi_t \in H_t \}$ is
equipped with a structure of standard Borel space that is compatible
with the projection $p : H \mapsto (0, \infty)$, given by $p((t, \xi_t) = t$.
\end{Definition}
Given an equi-modular $E_0$-semigroup $\alpha = \{\alpha_t: t \geq 0\}$ on a factorial non-commutative 
probability space $(M,\phi)$, we can always associate a super product system corresponding to the 
$E_0$-semigroup $\alpha$.
Assume $M$ is
acting standardly on $ \CH = L ^{2} ( M, \phi )$.
We consider, for every $t > 0$, the interwiner space
\be\label{Ealpha}
E^{\alpha_t} = \{ T\in \CB(L^2(M,\phi)) :  \alpha_t(x)T = Tx, \forall x \in M \}. 
\ee
$
\alpha^\prime = \{ \alpha^\prime_t = \jmath \circ \alpha_{t} \circ \jmath
\text{ : } t \geq 0 \} $
 defines an $E_0$-semigroup on the commutant $M^{\prime}$;
and similarly we have $E^{\alpha^{\prime}_t}$.

We first focus on the `fundamental unit' $\{u_t: t \ge 0\}$ - which 
will establish the fact that $E^{\alpha_t} \cap E^{\alpha^\prime_t} 
\neq \emptyset ~\forall t \geq 0$.  For $t \geq 0$, the fact that `$\phi$' 
is preserved by $\alpha_t$ implies the existence of a unique family 
(necessarily a one-parameter semigroup) $\{u_t: t \geq 0\}$ of 
isometries on $L^2(M)$ such that $u_t x \hat{1} = \alpha_t(x) \hat{1} ~\forall x \in M$,
 and consequently $u_t \in E^{\alpha_t}$.
As $\alpha_t$ is a equi-modular *-endomorphism of $M$, it follows - see
\cite{BISS} - that $u_t$ also `implements $\alpha'_t$, i.e., 
also $u_t x^\prime \hat{1} = \alpha_t^\prime(x^\prime) \hat{1} ~\forall x^\prime \in M^\prime$, 
and consequently that  $u_t \in E^{\alpha^\prime_t} $. Thus,
\be \label{nonnull}
u_t \in E^{\alpha_t} \cap E^{\alpha^\prime_t} ~\forall t \geq 0.
\ee

Now for every $t > 0 $, let us write $H(t) = E^{\alpha_t}\cap E^{\alpha^\prime_t}$.
In fact, $H(t)$ is actually a Hilbert space; if  $ S , T \in H(t) $,
then
\[T^*S \in (E^{\alpha_t})^* E^{\alpha_t} \cap
(E^{\alpha^\prime_t})^*E^{\alpha^\prime_t} \subset M^\prime \cap M = \C\]
and we find that  $T^*S$ is a scalar multiple of the identity and the value of 
that scalar defines an inner product by way of 
\[
 T^*S = \langle S , T \rangle  I . 
\]

Now $H = \{ (t , \xi_t ) : \xi_t \in H_t \}$ is a
super product  system with the family of isometries 
\[
 U_{s,t} :H_s\otimes  H_t \mapsto H_{s+t},
\]
uniquely determined by $U_{s,t} ( S\otimes T) = ST$, 
for $S\in H_s , T \in H_t $.

We collect the following explicit description of these intertwiner spaces which will 
useful in the sequel.

\begin{Theorem}\label{bimod}
$E^{\alpha_t}= [ M^\prime u_t ] = {\alpha_t(M)}^\prime u_t $.
\end{Theorem}

\begin{proof}
We  know that $E^{\alpha_t}$-see \cite{Alev}-  is a Hilbert 
von Neumannn $M^\prime -M^\prime$-bimodule. In particular  $E^{\alpha_t}$
is Hilbert von Neumann $M^\prime$ module. 
Now we shall verify that $[ M^\prime u_t ]$ is 
Hilbert von Neumann submodule of $E^{\alpha_t}$.
For that we need to check that $[ M^\prime u_t ]$ is Hilbert von Neumann $M^\prime$ 
module and $[ M^\prime u_t ] \subset E^{\alpha_t}$. For the first assertion 
notice that
\beast
[\{(m_1' u_t)^*m_2' u_t  : m_1', m_2' \in M^\prime \}] &=&
[u_t^*m_1^{'*}m_2'u_t]\\ 
&=& [u_t^*M^\prime u_t], \\
\eeast
so it suffices to check that $u_t^*M^\prime u_t \subset M'$, i.e.,
that $u_t^*m'u_t x = x u_t^*m'u_t ~\forall m' \in M', x \in M$; but
\beast
u_t^*m'u_t x &=& u_t^*m' \alpha_t(x)u_t ~~ (\text{ s
ince } u_t \in E_t^{\alpha})\\
&=& u_t^* \alpha_t(x)m'u_t\\
&=& (\alpha_t(x^*) u_t)^*m'u_t\\
&=&(u_tx^*)^*m'u_t\\
&=& xu_t^*m'u_t.
\eeast

Conversely, $m' = u_t^*u_t m^\prime = u_t^*\alpha_t^\prime
(m^\prime)u_t$ so $M' \subset u_t^*\alpha_t^\prime(M^\prime)u_t
\subset u_t^*M'u_t$, and hence we do have $M' = u_t^*M^\prime u_t$.

For the second  assertion observe that 
\beast 
\alpha_t (m) m^\prime u_t & =&  m^\prime \alpha_t(m)u_t\\
&=&  m^\prime u_t m,
\eeast
for all $ m \in M$ and $m^\prime \in M^\prime $, thus showing that 
$M^\prime u_t \subset E^{\alpha_t}$, and hence also that $[M^\prime u_t] \subset E^{\alpha_t}$. 

\bigskip 
Now suppose that there exist 
$T \in  E^{\alpha_t}$ such  that 
$T \in [M^\prime u_t]^\perp$, i.e., $T^*m^\prime u_t = 0 $ 
for all $m^\prime \in M^\prime $. Now notice that 
$ T^*m^\prime u_t \hat{1_M} = T^*m^\prime \hat{1_M} = T^*
\hat{m^\prime}$, and hence conclude that 
$T= 0$. Deduce then from the Riesz lemma that  
$E^{\alpha_t} = [M^\prime u_t]$.

Observe next that for $m \in M$ and $ x \in \alpha_t(M)^\prime$, we have
\beast
\alpha_t(m)xu_t &=& x\alpha_t(m)u_t\\
&=& xu_t m~,
\eeast
and deduce that 
$\alpha_t(M)^\prime u_t \subset E^{\alpha_t} $. On the other if $ T \in E^{\alpha_t} $
observe that 
\beast
T &=& Tu_t^* u_t\\
&=& yu_t
\eeast
where $ y = Tu_t^* \subset [E^{\alpha_t} {E^{\alpha_t}}^*] = \alpha_t(M)^\prime $.
That is $T \in \alpha_t(M)^\prime u_t$. So we have $E^{\alpha_t} \subset
\alpha_t(M)^\prime u_t$, yielding
$ E^{\alpha_t} = \alpha_t(M)^\prime u_t$, as desired.
\end{proof}

\begin{Remark}
\ben
\item We have already seen that $E^{\alpha_t}$ is Hilbert 
von Neumann $M^\prime-M^\prime$-bimmodule, so $[M^\prime u_t]$ and $\alpha_t(M)^\prime u_t$
are also 
a Hilbert von Neumann $M^\prime-M^\prime $-bimodule.
\item Replacing $\alpha^\prime$ by $\alpha$ and $M^\prime $ by $M$ in
  Proposition \ref{bimod}, we get $E_{t}^{\alpha^\prime} = [M u_t] =
  {\alpha_t^\prime(M^\prime)}^\prime u_t$. 
\een
\end{Remark}
Let $P(t) = \alpha_t(M)$ and $P_1(t)$ is the Jones basic construction. 
Then we have  $P_1(t) = \CJ P(t)^\prime \CJ$ (see \cite{BISS}).
But we know that ${\alpha_t^\prime(M^\prime)}^\prime =\CJ\alpha_t(M)^\prime \CJ =  P_1(t)$,
we may summarize thus:
\be\label{1}
E^{\alpha_t}= [ M^\prime u_t ] = {P(t)}^\prime u_t, 
\ee
and
\be\label{2}
E^{\alpha_t^\prime} = [M u_t] = P_1(t) u_t.
\ee

\section{CAR Flow} 
Let $\CH = L^2(0, \infty)\otimes \CK$, where $\CK$ is 
any Hilbert space.
Let $\CF_- (\CH)$ denote the anti-symmetric Fock space. For given  $ f \in \CH$, let $a(f)$ be the 
creation operator in  $\CB(\CF_- (\CH))$; thus:
\begin{enumerate}
\item $\CH \ni f \mapsto a(f)$ is $\C$-linear, 
\item {\bf (CAR)} \\
$ a(f)a(g) + a(g)a(f) = 0 \text{ and } a(f)a(g)^* + a(g)^*a(f) = \langle f, g\rangle 1,$
\end{enumerate}
where $f, g \in \CH$. Let $\CA$  be the unital $C^*$-algebra generated
by  $\{a(f): f \in \CH\}$ in $\CB(\CF_- (\CH))$. We note that 
$||a(f)|| = ||f||$ for $f \in \CH$.
 Now suppose $R \in \CB(\CH)$ satisfies $0 \leq R \leq 1$,
where of course $1$ is the identity operator  $id_{\CH}$.
The operator $R$ determines the so-called quasi-free state
$\omega_R$ on $\CA$ which satisfies the condition:
\[
\omega_R (a^*(f_m ) \cdots a^*(f_1)a(g_1 ) \cdots a(g_n )) = \delta_{mn} \text{det} (\langle Rg_i, f_j \rangle ).
\]
It is known - see \cite{BrRo}, \cite{G} - that there exists a representation
$\pi_R$ of the $C^*$-algebra $\CA$ on the Hilbert space
$\CH_R = \CF_-(\CH) \otimes \CF_-(\CH)$ defined by the formulae 
\begin{align*}
 &\pi_R (a(f)) = a((1-R)^{1/2}f) \otimes \Gamma +  1 \otimes a^* (qR^{1/2}f), \\
 &\pi_R (a^*(f)) = a^*((1-R)^{1/2}f) \otimes \Gamma +  1 \otimes a (qR^{1/2}f), \\
&\pi _R(1) = 1,
\end{align*}
 where $ f \in \CH$.              
Here $\Omega$ is the `vacuum vector' for the antisymmetric Fock space
$\CF_- (\CH)$, $q$ is an anti-unitary operator on $\CH$ with $q^2=1$, and 
$\Gamma$ is the unique unitary operator on $\CF_- (\CH)$ satisfying
the conditions $\Gamma a(f) = -a(f)\Gamma, f \in \CH$, and $\Gamma
\Omega = \Omega $.  
In this representation,
the state  $\omega_R$  becomes the vector state 
\[
\omega_R (x) = \langle \Omega \otimes \Omega, \pi_R(x)\Omega \otimes \Omega\rangle,
\]
for $x \in \CA$, and  $\CH_R =  \CF_-(\CH) \otimes \CF_-(\CH) = \overline{\pi_R(\CA) \Omega \otimes \Omega} $ 
becomes the GNS Hilbert space, under the assumption that both $R$ and $1-R$ are injective 
(and hence also have dense range). So $( \pi_R , \CH_R , \Omega \otimes \Omega )$ is 
the GNS triple for the $C^*$-algebra $\CA$ with respect to the state $\omega_R $.
We write $M_R = \{ \pi_R(\CA) \}^{\prime \prime }$, which is always a
factor, most often of type III (see \cite{PoSt} Theorem 5.1 and Lemma
5.3).

\bigskip 
Let $\{s_t\}_{t\geq}$ be the shift semigroup on $\CH$. Assume 
$s_t^*Rs_t = R$ for all $t \geq 0$. Then, by \cite{Arv} Proposition
13.2.3 and \cite{PoSt} Lemma 5.3, there exists an $E_0$-semigroup
$ \alpha = \{ \alpha_t : t\geq 0 \} $ on $ M_R  $, where $\alpha_t$ is uniquely 
determined by the following condition:
\[
 \alpha_t( \pi_R(a(f)) = \pi_R(a(s_tf)),
\]
for all $ f \in \CH, t\geq 0 $. This $E_0$-semigroup is called the {\bf CAR flow of 
 rank dim $\CK$ } (on $M_R$).\\ 

\subsection{Extendability of CAR flows}
For the remainder of this paper, we shall assume the following:
\ben
\item
$qs_t = s_tq$ for all $t\geq 0$. (Such a $q$ always exists.)  
\item We write $a_R (f)$ for $\pi_R (a(f))$ whenever $f\in \CH$, and
  write $\CJ$ for the modular conjugation operator of $M_R$. 
\item Both $R$ and $1-R$ are invertible; i.e., $\exists \e > 0$ such
  that $\e \leq R \leq 1-\e$.
\item $R$ is diagonalisable; in fact, there exists an orthonormal
  basis $\{f_i\}$ for $\CK$ with $Rf_i = \lambda_i f_i$ for some
  $\lambda_i \in [\e, 1-\e] \setminus \{{\half\}}$.
\item $Rs_t = s_tR ~\forall t \geq 0$. (Clearly then, also the
  Toeplitz condition $s_t^*Rs_t = R$ is met.)
\een

As we are unaware of whether, and if so where, these details may be found in the literature, 
we shall explicitly determine the modular operators in this case, and eventually ascertain 
(in Remark \ref{carem}) the equi-modularity of the CAR flow.

For any (usually orthonormal)set $\{w_i\}_{i\in \N}$ in $\CH$,
we shall use the following notation for the rest of the paper:
if $ I = ( i_1, i_2, \cdots, i_n ) $ and $J= ( j_1, j_2,   \cdots, j_m
)$ are ordered subsets of $\N$, then 
\ben 
\item $w_I = w_{i_1}\wedge \cdots \wedge w_{i_n}$,
\item $w_{IJ} = w_{i_1}\wedge \cdots \wedge w_{i_n} \wedge
  w_{j_1}\wedge \cdots \wedge w_{j_m}$,
\item  $Tw_I = Tw_{i_1}\wedge \cdots \wedge Tw_{i_n}$  for any operator $T \in \CB(\CH)$;
\item $\tilde{I} = \{i_n, \cdots, i_1\}$ so $w_{\tilde{I}} = w_{i_n}\wedge \cdots \wedge w_{i_1}$;
\item $a_R(w_I) =  a_R(w_{i_1})\cdots a_R(w_{i_n})$,\\
\item $a_R^*(f) = (a_R(f))^*$, so $a^*_R(w_{\tilde{I}}) =:  a^*_R(w_{i_n})\cdots a_R^*(w_{i_1}) =: (a_R(w_I))^*$,
\een

For a while, to simplify the notations, we write 
$ A = (1-R)^{1/2}, ~B = q R^{1/2}$ and notice that 
\beast 
\langle Bh_i, Bh_j \rangle &=& \langle q R^{1/2} h_i,q R^{1/2} h_j \rangle \\
&=& \langle R^{1/2} h_j,R^{1/2} h_i \rangle ~~\text{ since $q$ is anti-unitary}\\
&=& \langle R h_j, h_i \rangle\\
&=& \delta_{i,j} \lambda_i~.
\eeast

Now we write the following Lemmas without proof. Most of the proof follows from the use 
induction of cardinality of $L$ and our strong Toeplitz assumption 
(that $s_t$ commutes with $R$ and hence also with $A$ and $B$)

\begin{Lemma}\label{vect1}
Let $ L = \{ l_1 < \cdots <l_p \}$ be an ordered subset\footnote{For
  us, an ordered subset of $\N$ will always mean a finite subset of
  $\N$ with elements ordered in increasing order} of $\N$. Then we have
\be\label{l1}
a_R(h_L)  a^*_R(h_{\tilde{L}})\Omega \otimes \Omega
 = \sum c(L_1) Ah_{L_1} \otimes Bh_{L_1},
\ee
where the summation is taken over 
all ordered (possibly empty) subsets $L_1$ of $L$ and the $c(L_1)$ are
all non-zero real numbers - with $Ah_\emptyset$ and $Bh_\emptyset$ being
interpreted as $\Omega$. 
\end{Lemma}

\begin{Lemma}\label{vect2}
 Let $ L = \{ l_1 < \cdots <l_p \}$ so that, by Lemma \ref{vect1},
 equation \ref{l1} is satisfied. Then we have
\beast
 &(i)& a_R(s_th_L)  a^*_R(s_th_{\tilde{L}})\Omega\otimes \Omega = \sum c(L_1) As_th_{L_1} \otimes Bs_th_{L_1}, \forall t \geq 0;\\
&(ii)&a_R(h_I) a_R(h_L)  a^*_R(h_{\tilde{L}}) a^*_R(h_{\tilde{J}}) \Omega\otimes \Omega\\ 
&=& \sum (-1)^{|I||J|+ |L_1|(|I|+|J|)} c(L_1)  Ah_I \wedge Ah_{L_1} \otimes Bh_{L_1} \wedge Bh_J\\
&(iii)& a_R(s_th_I) a_R(s_th_L)  a^*_R(s_th_{\tilde{L}}) a^*_R(s_th_{\tilde{J}}) \Omega\otimes \Omega\\ 
&=& \sum (-1)^{|I||J|+ |L_1|(|I|+|J|)} c(L_1)  As_th_I \wedge
As_th_{L_1} \otimes Bs_th_{L_1} \wedge Bh_J , ~\forall t \geq 0
\eeast
where  $I$ and $J$ are finite ordered subsets of $\N$
with $I\cap J= I\cap L = L\cap J =\phi$, and the summation is taken over 
all ordered subsets $L_1$ of $L$ .  
\end{Lemma}

The fact that $s_t$ commutes with $R$ is seen to imply that the state $\omega_R$ is preserved by the CAR flow $\{\alpha_t: t \geq 0\}$ and hence there exists a canonical semi-group $\{S_t: t \geq 0\}$of isometries on $\CH_R$ such that
\[S_t(x(\Omega \otimes \Omega) = \alpha_t(x)(\Omega \otimes \Omega)  ~\forall x \in M_R.\]
The next lemma relates this semigroup $\{S_t: t \geq 0\}$ of isometries 
on $\CH_R$ and the shift semigroup $\{s_t: t \geq 0\}$ of isometries on $\CH$.

\begin{Lemma}\label{fnd}
 Let $\{h_i\}_{i\in \N}$ be the orthonormal basis  of $\CH$ as above. Then for every $t\geq0$, we have,  
\[
S_t ( h_L \wedge h_I \otimes qh_L \wedge qh_J )=
s_th_L \wedge s_th_I \otimes qs_th_L \wedge qs_t h_J,
\]
where  $I, J ,$ and $L$ are ordered subsets of $\N$
with $I\cap J= I\cap L = L\cap J =\phi$.
\end{Lemma}

Now the following lemma describes 
the action of the modular conjugation $\CJ$ and the commutant of $M_R$.
\begin{Lemma}\label{mod}
With the above notation, 
\beast\label{jno}
&(i)& \CJ ( h_I \wedge h_L \otimes qh_L \wedge qh_J )= h_{\tilde{J}} \wedge h_L \otimes q h_L \wedge qh_{\tilde{I}}\\
&(ii)&\CJ M_R \CJ = M_R^\prime = 
\{ \Gamma \otimes \Gamma b_R(h_i), b^*_R(h_j) \Gamma \otimes \Gamma : i,j \in \N \}^{\prime \prime}\\
&(iii)&\CJ a_R(h_l)\CJ = \Gamma \otimes \Gamma b_R^*(h_l)
\eeast
where ~~$b_R (h) = a(R^{1/2}h) \otimes \Gamma - 1 \otimes a^* (q {(1-R)}^{1/2}h)$.
\end{Lemma}
\begin{proof}
Recall the definition of the anti-linear (Tomita) operator $S$, given by
$Sx\Omega \otimes \Omega = x^* \Omega \otimes \Omega$ , $x \in M_R$.
We want to show the following expression for S:
\begin{align}\label{indst}
S(Ah_I &\wedge Ah_L \otimes B h_L \wedge Bh_J) \nonumber\\
&= Ah_{\tilde{J}} \wedge Ah_L \otimes Bh_L \wedge B h_{\tilde{I}}
\end{align}
The proof is again by induction on the cardinality of L. For $|L|= 0$,
the above assertion follows from
\begin{align*}
S( (1-R)^{1/2}h_I & \otimes qR^{1/2} h_J)\\
&= Sa_R(h_I)a_R^*(h_{\tilde{J}})(\Omega \otimes \Omega)\\
&= a_R(h_{\tilde{J}})a_R^*(h_I)(\Omega \otimes \Omega)\\
&= (1-R)^{1/2}h_{\tilde{J}} \otimes qR^{1/2} h_{\tilde{I}}
\end{align*}
 
Assume now that $|L|=n$ and that we know the validity of equation
\ref{indst} whenever $|L| < 1$.

The point to be noticed is that Corollary \ref{vect2}(ii) may be
re-written - in view of (i) each $c(L_1)$ (and $c(L)$ in particular)
being non-zero- as:

\begin{align}
Ah_I &\wedge Ah_L \otimes Bh_L \wedge B h_J\label{lhs}\\
&= d a_R( h_I) a_R(  h_L) a^*_R( h_{\tilde{L}}) a^*_R(  h_{\tilde{J}})
\Omega\otimes \Omega \label{rhs1}\\
& + \sum_{L_1 \subsetneq L} d (L_1)  Ah_I \wedge Ah_{L_1} \otimes \  Bh_{L_1} \wedge Bh_J\label{rhs2},
\end{align}
where the constants $d, d(L_1)$ are all real and remain unchanged
under changing $(I,J)$ to $(\tilde{J},\tilde{I})$.

Now apply $S$ to both sides of the above equation. Then the two terms
on the right side get replaced by the terms obtained by replacing
$(I,J)$ by $(\tilde{J},\tilde{I})$ (\ref{rhs1} by definition of $S$
and \ref{rhs2} by the induction hypothesis regarding \ref{indst}),
thereby completing the proof of equation \ref{indst}.

Equation (\ref{indst}) clearly implies that
\be \label{sond}
S(h_I \otimes qh_J) = \left((1-R)R^{-1}\right)^{\half} h_{\tilde{J}} \otimes
  q \left( (R(1-R)^{-1}\right)^\half h_{\tilde{I}}
\ee
(even if $I \cap J \neq \emptyset$; consideration of their
intersections was needed essentially in order to establish Lemma
\ref{vect1} and thereby deduce the foregoing conclusions.)

Let $\CD$ be the linear subspace spanned by $\{h_I \otimes qh_J: |I|,|J| \geq 0\}$. Thus $\CD$ is an obviously dense subspace of $\CH_R$ which
is contained in the domain of the Tomita conjugation operator $S$, where
its action is given by equation \ref{sond}. We now wish to show that
$\CD$ is also contained in $dom(S^*)$ and that $S^*|_{\CD}$ is the
operator $F$ defined by the equation

\be\label{fond}
F(h_I \otimes qh_J) = \left((R(1-R)^{-1}\right)^\half h_{\tilde{J}} \otimes
  q \left((1-R)R^{-1}\right)^\half h_{\tilde{I}}
\ee

Indeed, notice that
\begin{align*}
& \langle S(h_I \otimes q h_J, h_{I'} \otimes q h_{J'}\rangle\\
&=\langle \left((1-R)R^{-1}\right)^\half h_{\tilde{J}} \otimes
  q \left(R(1-R)^{-1}\right)^\half h_{\tilde{I}}, h_{I'} \otimes q
  h_{J'}\rangle\\ 
&= \langle \left((1-R)R^{-1}\right)^\half h_{\tilde{J}} , h_{I'}
\rangle \langle 
  q \left(R(1-R)^{-1}\right)^\half h_{\tilde{I}},  q  h_{J'}\rangle\\
&= \langle \left( (1-R)R^{-1}\right)^\half h_J , h_{\tilde{I'}}
\rangle \langle h_{\tilde{J'}}  ,  \left(R(1-R)^{-1}\right)^\half h_I\rangle\\
&=\langle \left(R(1-R)^{-1}\right)^\half h_{\tilde{J'}}  ,  h_I\rangle
\langle q \left( (1-R)R^{-1}\right)^\half h_ {\tilde{I'}}, qh_J
\rangle\\
&= \langle \left(R(1-R)^{-1}\right)^\half h_{\tilde{J'}} \otimes q \left( (1-R)R^{-1}\right)^\half h_ {\tilde{I'}}, 
h_I \otimes qh_J \rangle\\
&= \langle F ( h_{I'} \otimes q h_{J'} ), h_I \otimes qh_J \rangle\\
\end{align*}

Then, as $S$ and $F$ leave $\CD$ invariant, we see that
\begin{align*}
 FS &(h_I \otimes q h_J )\\
&= F ( \left((1-R)R^{-1}\right)^\half h_{\tilde{J}}) \otimes q \left(R(1-R)^{-1}\right)^\half h_{\tilde{I}})\\
&= R(1-R)^{-1} h_I \otimes
  q (1-R)R^{-1} h_J
\end{align*}

If $ S = \CJ \Delta^{1/2}$ is its polar  decomposition, with $\CJ$ the modular conjugation and $\Delta$ the  
modular operator for $M_R$, the action of $\CJ$ and $\Delta$ on $\CD$ are
thus seen to be given by the following rules respectively:
\[
\CJ ( h_I \wedge h_L \otimes qh_L \wedge qh_J )= h_{\tilde{J}} \wedge h_L \otimes q h_L \wedge qh_{\tilde{I}}
\]
\[\CJ( \Omega \otimes \Omega) = \Omega \otimes \Omega = \Delta(\Omega \otimes \Omega) \] and

\begin{align*}
\Delta& ( h_I \wedge h_L \otimes qh_L \wedge qh_J)\\
&=  R (1-R)^{-1} h_I \wedge  R (1-R)^{-1} h_L \otimes q (1-R)R^{-1}) h_L \wedge q(1-R)R^{-1}h_{J}\\
\end{align*}
This proves part (i) of the Lemma, while the proof of parts (ii) and (iii) only involve  of   the following facts:
\ben
\item Lemma \ref{vect1} and Corollary \ref{vect2} imply that $M_R(\Omega \otimes \Omega)$ is dense in $
 \CF_-(\CH) \otimes \CF_-(\CH)$;
\item Lemma \ref{mod} (i) implies that $\CJ(\CD) = \CD$
\item A painful but not difficult case-by-case computation reveals that 
\[\CJ a_R(f) \CJ = (\Gamma \otimes \Gamma) b_R^*(f) \in M_R^\prime ~\forall f \]
\een
\end{proof}

\begin{Remark}\label{carem}
 Using the definition of $S_t$ and $\CJ$ , it easily follows that 
$S_t \CJ = \CJ S_t$ for all $ t \geq 0$, which implies that $\alpha_t $ is 
equi-modular endomorphism for every $t \geq 0 $ . So now we are in the perfect situation
to talk about the extendability of the {\bf CAR flow} and under the above assumptions
on $R$, we prove that {\bf CAR flows} are not extendable. 
\end{Remark}

Now our aim is to explicitly determine $( \alpha_t(M_R)' \cap M_R )(\Omega\otimes\Omega)$ for the CAR flow $ \alpha = \{ \alpha_t : t\geq 0 \} $.

\bigskip
Let $\CP$ and $\CF$ denote copies of $\N$ - where we wish to think of
$\CF$ and $\CP$ as signifying the future and past respectively. 
Let us write $f_i = s_th_i$, so $\{ f_j \}_{ j\in \CF} $ is an
orthonormal basis for  $L^2(t, \infty) \otimes \CK$.
Also consider an orthonormal basis $\{ e_i \}_{ i\in \CP } $ of $L^2(0, t)
\otimes \CK$.
Then clearly 
$\{ e_i \}_{ i\in \CP} \cup \{ f_j \}_{ j\in \CF } $ is an
orthonormal basis for  $L^2(0, \infty) \otimes \CK$.

Let $F(\CF)$ and $F(\CP)$  denote the collections of all finite ordered subsets 
of $\CF$ and $\CP$ respectively. Then 
$\CL = \{ v_{I_1J_1} \otimes qv_{I_2J_2}: I_1, I_2 \in F(\CP), J_1, J_2 \in F(\CF)\}$
is an orthonormal basis for $\CF_-(\CH) \otimes \CF_-(\CH)$,
where $ v_{IJ} = e_{i_1}\wedge e_{i_2}\wedge \cdots \wedge e_{i_n}\wedge f_{j_1}\wedge f_{j_2}\wedge \cdots \wedge f_{j_m}$,
with $I = \{i_1<i_2 \cdots <i_n \} \subset \CP$ and $J=\{j_1 < j_2  \cdots j_m \subset \CF\}$.

\bigskip
Now if $T \in \CB(\CH_R)$, we will be working with the expansion of
$T(\Omega \otimes \Omega)$  with respect to above orthonormal basis.
Let us fix an $l \in \CF$. We shall write
$T(\Omega \otimes \Omega)$  in the following fashion, paying special
attention to the occurrence or not of $l$ in the first and/or second
tensor factor: 
\begin{align}\label{pudec}
T(\Omega \otimes \Omega) &=
 \sum (p_{00}(I_1J_1,I_2J_2)~ v_{I_1J_1}\otimes qv_{I_2J_2}\nonumber\\
&~~~+ p_{11}(I_1J_1,I_2J_2)~f_l\wedge v_{I_1J_1}\otimes qf_l \wedge
 qv_{I_2J_2}) \nonumber \\
                 & + \sum u_{00}(I_1J_1,I_2J_2)~ v_{I_1J_1}\otimes qv_{I_2J_2}\nonumber \\
                 &  +\sum u_{10}(I_1J_1,I_2J_2)~ f_l \wedge v_{I_1J_1}\otimes qv_{I_2J_2}\nonumber \\
                &   +\sum u_{01}(I_1J_1,I_2J_2)~ v_{I_1J_1}\otimes
                qf_l \wedge qv_{I_2J_2}\nonumber  \\
&  +\sum u_{11}(I_1J_1,I_2J_2)~ f_l \wedge v_{I_1J_1}\otimes qf_l \wedge qv_{I_{2}J_{2}}.
\end{align}

\noindent
Here and in the sequel, it will be tacitly assumed that the sums range
over $((I_1J_1), (I_2,J_2)) \in (F(\CP) \times F_l(\CF))^2 $ - where we write $F_l(\CF) = F(\CF\setminus \{l\})$ - and
$ p_{mn},u_{mn}: \{(I_1J_1,I_2J_2):I_k \in F(\CP), J_l \in F_l(\CF)\} \rar \C, m,n \in \{0,1\}$ where it is demanded that
$spt(p_{00}) = spt(p_{11})$ and that $spt(p_{11})$,
$spt(u_{00})$,$spt(u_{10})$, $spt(u_{01})$ and
$spt(u_{11})$ are all disjoint sets - where we write $spt(f)$ for the
subset of its domain where the function $f$ is non-zero. When
necessary to show their dependence on the index $l$, we shall anoint
these functions with an appropriate superscript, as in: $p^l_{11}(IJ,I'J')$.

The letters $p$ and $u$ are meant to signify `paired' and `unpaired'.
Thus, suppose $l \in \CF$,  $I, L \in F(\CP)$, and 
$J, K \in F_l(\CF)$. If both  $v_{IJ} \otimes qv_{LK}$
and  $ f_l \wedge v_{IJ} \otimes qf_l\wedge qv_{LK}$ 
appear in the representation of $T(\Omega \otimes \Omega)$ with non-zero
coefficients, then we shall think of $(IJ,KL)$ as being {\em an
  $l$-paired ordered pair}.
Thus $spt(p_{00})= spt(p_{11})$ is the collection of $l$-paired
ordered pairs, while $\cup_{m,n=0}^1spt(u_{mn})$ is the collection of $l$-unpaired ordered pairs.

Note that in such an expression of $T (\Omega \otimes \Omega)$ with respect to
different $l$, the type of a summand may change but the coefficients
remain the same up 
to sign, since two vectors anti-commute under wedge product.
We also note that $T (\Omega \otimes \Omega)$ has been
written with respect to the basis
$\CL'=\{ v_{I_1J_1} \otimes qv_{I_2J_2}, f_l \wedge v_{I_1J_1} \otimes qv_{I_2J_2},
v_{I_1J_1} \otimes qf_l \wedge qv_{I_2J_2}, 
f_l \wedge v_{I_1J_1} \otimes qf_l \wedge qv_{I_2J_2}  : I_i \in F(\CP), J_r \in F_l(\CF) \}$. 
There are five types of sums in the representation of 
$T(\Omega\otimes\Omega)$. For simplicity of  notation, let us write:
\beast
&(i)&
\xi_T(p)=  \sum (p_{00}(I_1J_1,I_2J_2)~ v_{I_1J_1}\otimes qv_{I_2J_2}\\
 &~~~~& + ~p_{11}(I_1J_1,I_2J_2)~f_l\wedge v_{I_1J_1}\otimes qf_l \wedge
 qv_{I_2J_2})\\
&(ii)& 
\xi_T(u_{00}) =  \sum u_{00}(I_1J_1,I_2J_2)~ v_{I_1J_1}\otimes qv_{I_2J_2}\\
&(iii)& \xi_T(u_{10}) =        \sum u_{10}(I_1J_1,I_2J_2)~ f_l \wedge v_{I_1J_1}\otimes qv_{I_2J_2}\\
&(iv)& \xi_T(u_{01}) =      \sum u_{01}(I_1J_1,I_2J_2)~ v_{I_1J_1}\otimes qf_l \wedge qv_{I_2J_2}\\
&(v)&\xi_T(u_{11}) = \sum u_{11}(I_1J_1,I_2J_2)~ f_l \wedge v_{I_1J_1}\otimes qf_l \wedge qv_{I_{2}J_{2}},
\eeast
and $\CS = \{p, u_{00}, u_{10}, u_{01}, u_{11} \}$. So we 
have:
\[
 T(\Omega \otimes \Omega ) = \sum_{ x \in \CS} \xi_T (x). 
\]
We also write:
\beast
&(i)&A_1 = \frac {1} {(1-\lambda_l)^{1/2}} a_R(f_l),\\
&(ii)&A_2 = \frac {-1} {{\lambda_l}^{1/2}}  \Gamma\otimes \Gamma b_R(f_l)\\
&(iii)&B_1 = \frac {1} {\lambda_l^{1/2}} a^*_R (f_l),~~ \\
&(iv)&B_2 = \frac {-1} {(1-\lambda_l)^{1/2}} b^*_R (f_l) (\Gamma \otimes \Gamma)\\
\eeast

There is an implicit dependence in the definition of the $A_i$'s and
$B_i$'s of the preceding equations on the arbitrarily chosen $l \in
\CF$. When we wish to make this dependence explicit (as in Theorem
\ref{prdtcsy} below), we shall adopt the following notational device:
$\CA_l = \{  A_1 , A_2 \}$ and
$\CB_l = \{ B_1, B_2 \}$. We shall
frequently use the following facts in the sequel: 
\ben \item $R^{1/2} f_l =
R^{1/2} s_th_l = {\lambda_l}^{1/2} f_l$;
\item $(1-R)^{1/2} f_l = (1-R)^{1/2} s_th_l = {(1-\lambda_l)}^{1/2}
  f_l$; and
\item $f_l \otimes \Omega, \Omega \otimes f_l \in ran(S_t) ~\forall ~l
  \in \CF$.
\een

\begin{Theorem}\label{prdtcsy}
If $T \in \CB(\CH_R)$ satisfies  $A_1T (\Omega\otimes\Omega) = A_2T (\Omega\otimes\Omega)$
and $B_1T (\Omega\otimes\Omega) = B_2T (\Omega\otimes\Omega)$ for
$A_1, A_2 \in \CA_l, B_1, B_2 \in \CB_l$ and for all $l \in \CF$, then 
\[
 T \Omega \otimes \Omega \subset 
[\{ v_{I_1} \otimes qv_{I_2} :  I_1, I_2 \in F(\CP),  (-1)^{|I_1|} = (-1)^{|I_2|} \} ],
\]
where $[~]$ denotes span closure. 
\end{Theorem}
\bigskip
We start with a $T \in \CB(\CH_R) $, which satisfies the hypothesis of the 
above Theorem \ref{prdtcsy} and write $ T(\Omega \otimes \Omega) $ as
in \ref{pudec} , for an arbitrary choice of index $l$. Then we
go through the following Lemmas and prove that 
the coefficient functions $ p_{00}, p_{11}, u_{10} , u_{01}, u_{11} $ are 
identically zero, while the support of $u_{00}$ is contained in the set 
$\{(I_1J_1,I_2J_2): J_1 \cup J_2 = \emptyset , (-1)^{|I_1|} =
(-1)^{|I_2|} \}$. The truth of this assertion for all choices of $l$
will prove our Theorem \ref{prdtcsy}.  We go through the following Lemmas regarding
the representation of $T(\Omega \otimes \Omega) $  whose proofs elementary and simle. We 
may sometime omit the details.

\begin{Lemma}\label{xiT}
Let $\eta(x)$ (resp., $\eta(y)$)  be a summand\footnote{By a summand of $\xi_T(p)$ we shall mean a `paired term' of
  the form $(p_{00}(I_1J_1,I_2J_2)~ v_{I_1J_1}\otimes qv_{I_2J_2}
  + ~p_{11}(I_1J_1,I_2J_2)~f_l\wedge v_{I_1J_1}\otimes qf_l \wedge
 qv_{I_2J_2})$ rather than an individual term of such a pair}
 of the sum $\xi_T(x)$ (resp, $\eta(y)$), where $ x,y \in \CS$. Then 
$\langle \eta(x), \eta(y) \rangle = 0 $ implies that 
$\langle X\eta(x), Y\eta(y) \rangle = 0 $, 
for all $x, y \in \CS$ and 
$X,Y \in \CA$ or $X,Y \in \CB$.
\end{Lemma} 
\begin{proof}
This follows from (i) the assumptions that $spt(p_{00}) =
spt(p_{11})$, 
(ii) $sp(p_{11})$,
$spt(u_{00})$,$spt(u_{10})$, $spt(u_{01})$ and
$spt(u_{11})$ are all disjoint sets and (iii) the definition of the action of $X, Y$ on $\eta(x)$.
\end{proof}

\begin{Lemma}\label{xiT1}
If~ $A_1 T(\Omega \otimes \Omega ) = A_2 T(\Omega \otimes \Omega )$,~ then 
$A_1 \xi_T (x) = A_2 \xi_T (x) $ for all $x \in \CS$. Similarly 
if ~ $B_1 T(\Omega \otimes \Omega ) = B_2 T(\Omega \otimes \Omega )$, then 
$B_1 \xi_T (x) = B_2 \xi_T (x) $ for all $x \in \CS$. 
\end{Lemma}
\begin{proof}
This follows from
\[\|(A_1 - A_2) T(\Omega \otimes \Omega)\|^2 = \sum_{x \in \CS} \|(A_1
- A_2) \xi_T(x)\|^2 ~,\]
which is a consequence of Lemma \ref{xiT}.
\end{proof}

Now onwards we assume that  $T$ satisfies the hypothesis of 
the Theorem \ref{prdtcsy} and with the foregoing notations we have the following Lemma regarding the 
coefficients of the representation of $T(\Omega\otimes\Omega)$.

\begin{Lemma}\label{coeff} $p_{00},~p_{11},~u_{00},~u_{10},~u_{01},~u_{11} $ satisfy
the following equations:
\beast
&(i)&\sigma (I_2,J_2) p_{00}(I_1J_1,I_2J_2) +   p_{11}(I_1J_1,I_2J_2)\frac {\lambda_l^{1/2}} {(1-\lambda_l)^{1/2}}\nonumber\\
 &=&  \sigma (I_1,J_1)  p_{00}(I_1J_1,I_2J_2) +  \rho (I_1J_1, I_2J_2)   p_{11}(I_1J_1,I_2J_2) \frac {(1-\lambda_l)^{1/2}} {\lambda_l^{1/2}},\\
&(ii)&
\sigma (I_2,J_2)  u_{00}(I_1J_1,I_2J_2) = \sigma (I_1,J_1)    u_{00}(I_1J_1,I_2J_2),\\
&(iii)&
{u_{01}(I_1J_1,I_2J_2)} \frac {\lambda_l^{1/2} } {(1-\lambda_l)^{1/2}} = \star u_{01}(I_1J_1,I_2J_2) \frac {(1-\lambda_l^{1/2}} {{\lambda_l})^{1/2}},\\
&(iv)&
\star u_{11}(I_1J_1,I_2J_2)\frac {\lambda_l^{1/2}} {(1-\lambda_l)^{1/2}} = \star u_{11}(I_1J_1,I_2J_2)\frac {(1-\lambda_l)^{1/2}} {{\lambda_l}^{1/2}},\\
&(v)& \star \frac {(1-\lambda_l)^{1/2}} {\lambda_l^{1/2}}u_{10}(I_1J_1,I_2J_2)
= \star \frac {\lambda_l^{1/2}} {(1-\lambda_l)^{1/2}} u_{10}(I_1J_1,I_2J_2)
\eeast
where $ \sigma : F(\CP) \times F(\CF\} \rar \{1, -1\}$ is defined by
$\sigma (I,J) = (-1)^{|I|+ |J|}$, $ \rho : \{(I_1J_1,I_2J_2):I_k \in F(\CP), J_l \in F(\CF) \} \rar \{1,-1\}$,
defined by $\rho (I_1J_1, I_2J_2) = (-1)^{|I_1|+ |J_1|+ |I_2|+ |J_2|}$, and   
$ \star = \pm 1$. 
\end{Lemma}

\begin{proof}

$T$ satisfies $A_1 (T\Omega\otimes \Omega) =A_2 T(\Omega\otimes \Omega )$. So from the 
Lemma \ref{xiT1}, we have $A_1 \xi_T (x) = A_2 \xi_T(x)$ foll $x \in \CS$.
Now for every $x\in \CS$, we  separately compute $A_1 \xi_T (x)$ and $A_2 \xi_T(x)$ and 
compare their coefficients.

\bigskip \noindent $(i)$  If $A_1 \xi_T (p) = A_2 \xi_T(p)$,
observe that 
\begin{align*}
A_1 \xi_T (p) &=\frac {1} {(1-\lambda_l)^{1/2}} a_R (f_l) \xi_T(p) \\
&=\sum (\frac {p_{00}(I_1J_1,I_2J_2)} {(1-\lambda_l)^{1/2}} a_R (f_l) v_{I_1J_1}\otimes qv_{I_2J_2}\\
& + \frac {  p_{11}(I_1J_1,I_2J_2)} {(1-\lambda_l)^{1/2}} a_R (f_l) 
f_l\wedge v_{I_1J_1}\otimes qf_l \wedge qv_{I_2J_2})\\ 
&=\sum (\sigma (I_2, J_2) p_{00}(I_1J_1,I_2J_2\\
&~~~+~p_{11}(I_1J_1,I_2J_2) \frac {\lambda_l^{1/2}} {(1-\lambda_l)^{1/2}} ) f_l\wedge v_{I_1J_1}\otimes qv_{I_2J_2}
\end{align*}
while
\begin{align*}
&A_2 \xi_T(p)\\
&= \frac {-1} {{\lambda_l}^{1/2}}  \Gamma\otimes \Gamma b(f_l) \xi_T(p)  \\
&=\sum (p_{00}(I_1J_1,I_2J_2)\frac {-1} {{\lambda_l}^{1/2}} 
\Gamma\otimes \Gamma b(f_l) v_{I_1J_1}\otimes qv_{I_2J_2}\\
& +   p_{11}(I_1J_1,I_2J_2) \frac {-1} {{\lambda_l}^{1/2}} \Gamma\otimes \Gamma b(f_l)f_l\wedge v_{I_1J_1}\otimes qf_l \wedge qv_{I_2J_2}\\
&= \sum ( \sigma (I_1, J_1) p_{00}(I_1J_1,I_2J_2)\\
&~~+ \rho (I_1J_1 , I_2J_2 )   p_{11}(I_1J_1,I_2J_2) 
\frac {(1-\lambda_l)^{1/2}} {\lambda_l^{1/2}} ) f_l\wedge v_{I_1J_1}\otimes  qv_{I_2J_2}\\
\end{align*}
and $(i)$ follows upon comparing coefficients in the two equations above.

\bigskip \noindent Equations $(ii)$, $(iii)$ and $(iv)$ are proved by arguing exactly as for $(ii)$ above.

\bigskip \noindent As for (v), we also have $ B_1 T(\Omega \otimes \Omega ) = B_2 T(\Omega \otimes \Omega )$. So from
Lemma \ref{xiT1}, we have $B_1 \xi_T(x) =B_2 \xi_T(x)$ for all $x\in \CS$. 
In particular we have $B_1 \xi_T(u_{10}) =B_2 \xi_T(u_{10})$. Then $v$ follows from the comparing coefficients
of $B_1 \xi_T(u_{10})$  and $B_2 \xi_T(u_{10})$.

\end{proof}

With foregoing notation and the assumptions on $T$, we have the following 
Corollary.
\begin{Corollary}\label{coeffz}
If we represent $T(\Omega \otimes \Omega ) $ las in eqn. \ref{pudec}, 
then 
\[
u_{01}= u_{11}= u_{10} = 0.
\]
That is the functions $u_{01}$,~$u_{11}$ and $u_{10}$ are identically zero.
\end{Corollary}
\begin{proof}
This is an immediate consequence of Lemma \ref{coeff}(iii), \ref{coeff}(iv) and \ref{coeff}(v) and the assumption that $\lambda_l \neq 1/2 ~\forall l$.
\end{proof}

 We continue to assume that an operator $T \in \CB(\CH_R)$ satisfies the hypothesis of the Theorem \ref{prdtcsy} and proceed to analyse the representation 
 of $T(\Omega\otimes \Omega)$ as in eqn. (\ref{pudec}). 
     
\begin{Remark}\label{l0}
\ben
\item Lemma \ref{coeff}(i)
 implies that if $(IJ,KL)$ are $l$-paired, then $\sigma(I,J) \neq \sigma(K,L)$. ({\em Reason:} Otherwise, since $\rho(IJ,KL) = \pm 1$, and $|p_{11}(IJ.KL)| \neq 0$, we must have $\lambda_l = \half$.) 
\item Lemma \ref{coeff}(ii) implies that 
if $u_{00}(I_1J_1,I_2J_2) \neq 0$, then $\sigma (I_2,J_2) = \sigma (I_1,J_1)$, i.e.
$(-1)^{|I_1| + |J_1| } =  (-1)^{|I_2| + |J_2| }$.
\een
\end{Remark}

Now we wish to compare the representations of $T(\Omega \otimes \Omega)$ for different $l$'s.

\begin{Lemma}\label{prd1}
 Let $I,K \in F(\CP)$ and $J, L \in F(\CF)$. If  
a term of the form $v_{IJ} \otimes qv_{KL}$ appears in 
$T(\Omega \otimes \Omega)$ with non-zero coefficient, then  $(IJ,KL)$
can be {\bf $w$-paired} for at most finitely many $w \in \CF$ with $w \notin J\cup L$.
\end{Lemma}

\begin{proof}
Suppose, if possible, that $\{l_n:n \in \CF\}$ is an infinite sequence of distinct indices such that $(IJ,KL)$ is $l_n$-paired for each $n \in \N$. Then we may, by Remark \ref{l0}(1), conclude that
$\{\sigma(I,J),\sigma(K,L)\} = \{1,-1\}$. 

Deduce now from Lemma \ref{coeff}(i) that
\begin{align}\label{wprd1}
&\sigma (K,L) p_{00}(IJ,LK) +    \star p^{l_n}_{11}(IJ,Kl)\frac {\lambda_{l_n}^{1/2}} {(1-\lambda_{l_n})^{1/2}}\nonumber\\
 &=  \sigma (I,J)  p_{00}(IJ,LK) +  \star  p^{l_n}_{11}(IJ,KL) \frac {(1-\lambda_{l_n})^{1/2}} {\lambda_{l_n}^{1/2}}
\end{align}
where $\star \in \{+,-\}$. Since $ \lambda_{l_n} \in  ( \e, 1- \e ) \setminus \{1/2\} $ for all $n$, we see that
$\{\frac {\lambda_{l_n}^{1/2}} {(1-\lambda_{l_n})^{1/2}} : n \in \N\} $  and 
 $\{ \frac {(1-\lambda_{l_n})^{1/2}} {\lambda_{l_n}^{1/2}} : n \in \N\} $ are
bounded sequences. As
$p^{l_n}_{11}(IJ,LK)$  are Fourier 
coefficients, the sequence $\{ p^{l_n}_{11}(IJ,LK) \} $ converges to $0$, as $ n \rightarrow \infty $.
Clearly then  $\{\frac {\lambda_{l_n}^{1/2}} {(1-\lambda_{l_n})^{1/2}}  p^{l_n}_{11}(IJ,LK) : n \in \N  \} $ and 
$\{ \frac {(1-\lambda_{l_n})^{1/2}} {\lambda_{l_n}^{1/2}}  p^{l_n}_{11}(IJ,LK) : n \in \N\}$ are sequences
converges to $0$, as $ n \rightarrow \infty $.
So from the above equation we get $ p_{00}(IJ,LK) =0  $. But we had assumed that $p_{00}(IJ,LK)$ is non-zero.
Hence $v_{IJ} \otimes qv_{KL}$ can not be $l$-paired for 
infinitely many $l \in \CF$ with $l \notin J\cup L$.
\end{proof}

\begin{Lemma}\label{prd2}
Let $I,K \in F(\CP)$ and $J, L \in F(\CF)$ with $l \notin J\cup L$. Suppose  
an element of the form $v_{IJ} \otimes qv_{KL}$, appearing in 
$T\Omega \otimes \Omega$ with a non-zero coefficient.
Then we have $(-1)^{|I| + |J| } = (-1)^{|K| + |L|}$.
\end{Lemma}
\begin{proof}
From Lemma \ref{prd1}, we can find a $l_0 \in \CF$ such that 
$l_0 \notin J\cup L $ and  $v_{IJ} \otimes qv_{KL}$ is not 
$l_0$-paired. If we write $T\Omega \otimes \Omega$ with respect to $l_0$,
we see that $v_{IJ} \otimes qv_{KL}$ appears with exactly the same coefficient as in the third type of sum.
So by observing the Remark \ref{l0} with respect to $l_0$,  see that $(-1)^{|I| + |J| } = (-1)^{|K| + |L|}$.
\end{proof}

Again with the foregoing notations, we have the following 
Lemma about the coefficients of the representation of $T(\Omega\otimes\Omega)$.
\begin{Lemma}\label{prd3}
 $p_{00} = p_{11} = 0 $. 
\end{Lemma}

\begin{proof}
Recall the equation \ref{coeff}(i) from Lemma \ref{coeff}:
\begin{align*}
&\sigma (I_2,J_2) p_{00}(I_1J_1,I_2J_2) +   p_{11}(I_1J_1,I_2J_2)\frac {\lambda_l^{1/2}} {(1-\lambda_l)^{1/2}}\\
 &=  \sigma (I_1,J_1)  p_{00}(I_1J_1,I_2J_2) +  \rho (I_1J_1, I_2J_2)   p_{11}(I_1J_1,I_2J_2) \frac {(1-\lambda_l)^{1/2}} {\lambda_l^{1/2}}
\end{align*}
where $\sigma (I_2,J_2) =  (-1)^{|I_2| + |J_2| }$ and $\sigma (I_1,J_1) = (-1)^{|I_1| + |J_1|}$.
But from \ref{prd2} we have
 $(-1)^{|I_2| + |J_2|} = (-1)^{|I_1| + |J_1|}$. Since $ \lambda_l \neq 1/2$, from 
the above equation we get $ p_{11}(I_1J_1,I_2J_2) = 0$, which implies that 
$ p_{00}(I_1J_1,I_2J_2) = 0$, since $spt(P_{00}) = spt(p_{11})$, i.e they 
have the same support. 
\end{proof}
\begin{Lemma}\label{prd4}
$T\Omega \otimes \Omega = \sum x(I_1 I_2)~ v_{I_1} \otimes qv_{I_2} $, where 
the summation is taken over $I_1, I_2  \in F(\CP)$ with $(-1)^{|I_1|} = (-1)^{|I_2|} $.
\end{Lemma}
\begin{proof}
So we started with a representation of $T\Omega \otimes \Omega$ like 
\ref{pudec} and by using the Corollary \ref{coeffz} and Lemma \ref{prd3} 
ended up with the following conclusions;
\[
p_{00} = p_{11} = u_{10}= u_{01} = u_{11} = 0.
\]
Thus finally the  representation of $T\Omega \otimes \Omega$ will be 
of the form  
\[
T(\Omega \otimes \Omega) =\sum u_{00}(I_1J_1,I_2J_2)~ v_{I_1J_1}\otimes qv_{I_2J_2},\\
\]
where the summation is taken over $I_1, I_2 \in F(\CP) $, $J_1, J_2 \in F(\CF)$ with 
$(-1)^{|I_2| + |J_2|} = (-1)^{|I_1| + |J_1|}$ and $ l \notin J_1 \cup J_2 $, for $l \in \CF$.
Since this is true for all $l \in \CF$, $J_1, J_2$  are empty sets, i.e.
 \[T(\Omega \otimes \Omega) = \sum x(I_1, I_2)~ v_{I_1} \otimes qv_{I_2},\] where
the summation is taken over $I_1, I_2  \in F(\CP)$ with $(-1)^{|I_1|} = (-1)^{|I_2|} $ and 
$x(I_1 ,I_2) = u_{00}(I_1 \emptyset,I_2 \emptyset)$
are complex numbers. 
\end{proof}
So finally the above Lemma \ref{prd4}  proves our theorem \ref{prdtcsy}.

\bigskip

\begin{Theorem}\label{prdtcsy1}
Let $T\in \alpha_t(M_R)^\prime \cap M_R$, then 
\[
 T( \Omega \otimes \Omega) \subset 
[\{ v_{I_1} \otimes qv_{I_2} :  I_1, I_2 \in F(\CP),  (-1)^{|I_1|} = (-1)^{|I_2|} \} ],
\]
\end{Theorem}
\begin{proof}
It is enough to prove that 
$A_1 T(\Omega\otimes\Omega) = A_2 T(\Omega\otimes\Omega)$ and 
$B_1 T(\Omega\otimes\Omega) = B_2 T(\Omega\otimes\Omega)$, then 
it follows from the Theorem \ref{prdtcsy1}.
\bigskip

Observe that, 
\begin{align*}
A_1& T(\Omega\otimes \Omega)\\
&= \frac {1} {(1-\lambda_l)^{1/2}}a_R(f_l)T\Omega\otimes \Omega \\
&= \frac {1} {(1-\lambda_l)^{1/2}} T a_R (f_l) \Omega\otimes \Omega~\text{ since } T \in \alpha_t(M_R)^\prime\\
&= T f_l \otimes \Omega \\
&= \frac {-1} {{\lambda_l}^{1/2}} T \Gamma\otimes \Gamma b_R(f_l) \Omega\otimes \Omega\\
&= \frac {-1} {{\lambda_l}^{1/2}}  \Gamma\otimes \Gamma b_R(f_l) T\Omega\otimes \Omega~
\text{ since } T \in M_R \text{ and } \Gamma\otimes \Gamma b_R(f_l)\in  M_R^\prime\\
&=A_2 T(\Omega\otimes \Omega).
\end{align*}

So we have $A_1 T(\Omega\otimes \Omega)  = A_2 T(\Omega\otimes \Omega)$.
Again observe that,  
\begin{align*}
B_1 &T(\Omega\otimes \Omega)\\
&= \frac {1} {{\lambda_l}^{1/2}}  a^*_R (f_l)T\Omega\otimes \Omega \\
&=\frac {1} {{\lambda_l}^{1/2}}T a^*_R (f_l)\Omega\otimes \Omega~\text{ since } T \in \alpha_t(M_R)^\prime\\
&= T\Omega \otimes f_l\\
&=  \frac {-1} {(1-\lambda_l)^{1/2}} T b^*_R (f_l)\Gamma \otimes \Gamma \Omega\otimes \Omega\\
&=  \frac {-1} {(1-\lambda_l)^{1/2}}  b^*_R (f_l)\Gamma \otimes \Gamma T\Omega\otimes \Omega,
\text{ since } T \in M_R \text{ and }  b^*_R(f_l)\Gamma\otimes \Gamma \in  M_R^\prime\\
&=B_2 T(\Omega\otimes \Omega).
\end{align*}
That is $B_1 T(\Omega\otimes \Omega) = B_2 T(\Omega\otimes \Omega)$.

\end{proof}

\begin{Theorem}\label{ntext}
CAR flow $\alpha = \{ \alpha_t : t \geq 0  \} $ is not not extendable. 
\end{Theorem}
\begin{proof}
 It is enough to show that for some $t >0 $, $\alpha_t : M_R \mapsto M_R $ is not 
extendable. To prove $\alpha_t$ is not extendable, we use the
Theorem 3.7 of \cite{BISS}.
\bigskip
We observe that
\beast
\lefteqn{[\{  y\alpha_t(x) \Omega \otimes \Omega : x \in M_R, y \in \alpha_t(M_R)^\prime \cap M_R \}]}\\
&=& [\{  \alpha_t(x)y \Omega \otimes \Omega : x \in M_R, y \in \alpha_t(M_R)^\prime \cap M_R \}]\\
&=& [\{  (a_R (f_J) a^*_R (f_L))^* T\Omega \otimes \Omega : J,L \in F(\CF), T \in \alpha_t(M_R)^\prime \cap M_R \}].
\eeast
Now if $T \in \alpha_t(M_R)^\prime \cap M_R $, then from the Theorem \ref{prdtcsy1}
we have,  $ T(\Omega \otimes \Omega) \in \{ v_{I_1} \otimes qv_{I_2} :  I_1, I_2 \in F(\CP),  
(-1)^{|I_1|} = (-1)^{|I_2|} \}$. 
If $ g\in \CP $, we notice that 
\begin{align*} 
\langle  (a_R (f_J) a^*_R (f_L))^* &T\Omega \otimes \Omega,~e_g \otimes \Omega \rangle\\
&= \langle   T \Omega \otimes \Omega,~ a_R (f_I) a^*_R ( f_L) e_g \otimes \Omega \rangle\\
&= 0.
\end{align*}
So from the above, we conclude that 
$\{  y\alpha_t(x) \Omega \otimes \Omega : x \in M_R, y \in \alpha_t(M_R)^\prime \cap M_R \} $
is orthogonal to the vector $e_g \otimes \Omega$, i.e. 
$\{  y\alpha_t(x) : x \in M_R, y \in \alpha_t(M_R)^\prime \cap M_R \} $ 
can not be weakly total in $M_R$, so by the Theorem 3.7 of \cite{BISS}  $\alpha_t $ 
can not be extendable. 
\end{proof}

\begin{Remark}
 It has been proved in section 5\cite{ABS} that CAR flows, arising from 
quasi-free state for scalar, on type $III$  factors are  extendable.
But we have prove that CAR arising from quasi-free state for diagonalisable 
positive contractions ( in particular scalars ) are not extendable.
So our result shows that there is some error in section 5 \cite{ABS} 
regarding the conclusion
of extendability of CAR flows. In-fact we think that there is a mistake 
in the theorem 4 of section 5 \cite{ABS} and for that their conclusion
regarding the extendability of CAR flows went wrong. 

\end{Remark}

\subsection{Super Product System for CAR flows}
Now we recall the definition of fundamental unit, Since 
for every $t\geq 0$, 
quasi-free state is invariant under $\alpha_t$, we get the fundamental unit. 
Let $\{S_t\}_{t \geq 0 } $ be the fundamental unit for the CAR  flow.
Then recall from the Theorem \ref{bimod} that the common 
intertwiner space for the CAR flow 
is $H_t = \alpha_t(M_R)^\prime S_t \bigcap \CJ \alpha_t(M_R)^\prime \CJ S_t$.
Now our aim is to find explicitly $H_t (\Omega\otimes\Omega)$.

\begin{Theorem}\label{prdtcsy1}
Let $T\in H_t $, 
\[
 T( \Omega \otimes \Omega) \subset 
[\{ v_{I_1} \otimes qv_{I_2} :  I_1, I_2 \in F(\CP),  (-1)^{|I_1|} = (-1)^{|I_2|} \} ],
\]
\end{Theorem}
\begin{proof}
It is enough to prove that 
$A_1 T(\Omega\otimes\Omega) = A_2 T(\Omega\otimes\Omega)$ and 
$B_1 T(\Omega\otimes\Omega) = B_2 T(\Omega\otimes\Omega)$, then 
it follows from the Theorem \ref{prdtcsy}.
 
For $ T \in H_t $, we may - by Theorem \ref{prdtcsy}, as the $u_t$ there
is our $S_t$, when $M=M_R$ and $\alpha$ is the CAR flow - write $T = T_1S_t
= T_2S_t $, where  
$T_1 \in \alpha_t(M_R)^\prime$ and $T_2 \in \CJ \alpha_t(M_R)^\prime \CJ $. 
As $T_1$ and $T_2$ agree on the range of $S_t$, observe that,
\[
T(\Omega \otimes \Omega) = T_1(\Omega \otimes \Omega) = T_2 (\Omega \otimes \Omega),
\]
since $ S_t (\Omega\otimes \Omega) = (\Omega\otimes \Omega)$. 
We note that, 
\begin{align*}
A_1 T(\Omega\otimes \Omega)&= \frac {1} {(1-\lambda_l)^{1/2}}a_R(f_l)T\Omega\otimes \Omega \\
&=  \frac {1} {(1-\lambda_l)^{1/2}} a_R(f_l)T_1\Omega\otimes \Omega\\
&= \frac {1} {(1-\lambda_l)^{1/2}} T_1 a_R (f_l) \Omega\otimes \Omega~\text{ since } T_1 \in \alpha_t(M_R)^\prime\\
&= T_1 f_l \otimes \Omega \\
&= T_2 f_l \otimes \Omega ~~ \text{ by (3) above, since } T_1S_t = T_2S_t \\
&= \frac {-1} {{\lambda_l}^{1/2}} T_2 \Gamma\otimes \Gamma b_R(f_l) \Omega\otimes \Omega\\
&= \frac {-1} {{\lambda_l}^{1/2}}  \Gamma\otimes \Gamma b_R(f_l) T_2\Omega\otimes \Omega~\text{ since } T_2 \in \CJ\alpha_t(M_R)^\prime \CJ\\
&= \frac {-1} {{\lambda_l}^{1/2}}  \Gamma\otimes \Gamma b_R(f_l) T\Omega\otimes \Omega\\
&=A_2 T(\Omega\otimes \Omega).
\end{align*}

So we have $A_1 T(\Omega\otimes \Omega)  = A_2 T(\Omega\otimes \Omega)$.
Again by similar kind of computation as above, we observe that
$B_1 T(\Omega\otimes \Omega) = B_2 T(\Omega\otimes \Omega)$.

\end{proof}

Let us recall that $\{h_i\}_{i\in \N}$ and $\{ e_i \}_{ i\in \CP }$ are 
orthonormal bases of $\CH = L^2(0, \infty ) \otimes \CK$ and $L^2(0, t) \otimes \CK$ respectively. 
Now with respect to the fix orthonormal basis   $\{ e_i \}_{ i\in \CP }$ of $L^2(0, t) \otimes \CK$, 
we define the following operator on $\CH_R$. If $I_1, I_2 \in F(\CP) $ with $(-1)^{|I_1|} = (-1)^{|I_2|}$, 
then there exists an operator $T_{I_1 I_2} : \CH_R \mapsto \CH_R $ which is
defined by the following rule,
\[
T_{I_1 I_2} ( h_{J_1} \otimes q h_{J_2}) = (-1)^{|I_1|~|J_1|} s_th_{J_1} \wedge e_{I_1} \otimes q s_t h_{J_2} \wedge q e_{I_2},
\]
where $J_1, J_2$ are finite ordered subset of $\N$. With the above notation we have the following theorem.
\begin{Theorem}
For $t >0 $, we have 
\[
H_t = [ \{ T_{I_1 I_2 } : I_1, I_2 \in F(\CP) , (-1)^{|I_1|} = (-1)^{|I_2|} \}].
\]
\end{Theorem}
\begin{proof}
If $I_1, I_2 \in F(\CP)$ with  $(-1)^{|I_1|} = (-1)^{|I_2|}$, we check that 
$T_{I_1 I_2 } \in H_t$. We want to prove $\alpha_t(a_R (h_l) ) T_{I_1 I_2} = T_{I_1 I_2} a_R(h_l)$
for all $ l \in \N$. 
We observe that 
\begin{align*}
\alpha_t(a_R (h_l) ) &T_{I_1 I_2} ( h_{J_1} \otimes q h_{J_2})\\
&= a_R (s_th_l) T_{I_1 I_2} ( h_{J_1} \otimes q h_{J_2})\\
&= (-1)^{|I_1||J_1|} a_R (s_th_l) ( s_th_{J_1} \wedge e_{I_1} \otimes q s_t h_{J_2} \wedge q e_{I_2})\\
&= \frac{(-1)^{|I_1||J_1|+ |J_2| + |I_2|}} {(1-\lambda_l)^{1/2}} (s_th_l \wedge s_th_{J_1} \wedge e_{I_1} \otimes q s_t h_{J_2} \wedge q e_{I_2})\\
&+ \frac{(-1)^{|I_1|~|J_1|}} {{\lambda_l}^{1/2}}  s_th_{J_1} \wedge e_{I_1} \otimes a^*(s_tqh_l)q s_t h_{J_2} \wedge q e_{I_2})\\
&= \frac{(-1)^{|I_1|~(|J_1|+ 1)+|J_2|}} {(1-\lambda_l)^{1/2}} (s_th_l \wedge s_th_{J_1} \wedge e_{I_1} \otimes q s_t h_{J_2} \wedge q e_{I_2})\\
&+ \frac{(-1)^{|I_1|~|J_1|}} {{\lambda_l}^{1/2}}  s_th_{J_1} \wedge e_{I_1} \otimes a^*(s_tqh_l)q s_t h_{J_2} \wedge q e_{I_2}),\\
\end{align*}
In the above to write second last to last equation, we have used $(-1)^{|I_1|} = (-)^{|I_2|}$. 
On the other hand observe that 
\begin{align*}
T_{I_1 I_2}& a_R (h_l)( h_{J_1} \otimes q h_{J_2})\\
&= T_{I_1 I_2} ((-1)^{|J_2|} (1-\lambda_l)^{1/2} (h_l \wedge h_{J_1} \otimes qh_{J_2})\\
&+ {\lambda_l}^{1/2} h_{J_1} \otimes a^*(qh_l)q h_{J_2} )\\
&= \frac{(-1)^{|I_1|~(|J_1|+ 1)+|J_2|}} {(1-\lambda_l)^{1/2}} (s_th_l \wedge s_th_{J_1} \wedge e_{I_1} \otimes q s_t h_{J_2} \wedge q e_{I_2})\\
&+ \frac{(-1)^{|I_1|~|J_1|}} {{\lambda_l}^{1/2}} s_th_{J_1} \wedge e_{I_1} \otimes a^*(s_tqh_l)q s_t h_{J_2} \wedge q e_{I_2})\\,
\end{align*}

Note that in the above equation we have used $s_t q = qs_t$. Finally above computations show that 
$\alpha_t(a_R (h_l) ) T_{I_1 I_2} = T_{I_1 I_2} a_R (h_l)$. Again similar computation will show that 
$\alpha_t(a^*_R (h_l) ) T_{I_1 I_2} = T_{I_1 I_2} a^*_R (h_l)$, for all $l \in \N$.
So we conclude that $T_{I_1I_2} \in  E^{\alpha_t} $.
From theorem \ref{bimod} we have
$ E^{\alpha_t} =\alpha_t(M)^\prime S_t $. So we have  $T_{I_1I_2} \in \alpha_t(M_R)^\prime S_t$.
Now recall the definition of $\CJ$ from lemma \ref{mod} and  observe that 
\beast
\CJ T_{I_1 I_2 } \CJ ( h_{J_1} \otimes q h_{J_2}) &=& 
\CJ T_{I_1 I_2 } ( q^2h_{\tilde{{J_2}}} \otimes q h_{\tilde{{J_1}}})\\
&=&\CJ T_{I_1 I_2 } ( h_{\tilde{{J_2}}} \otimes q h_{\tilde{{J_1}}})\\
&=& \CJ ((-1)^{|I_1|~|J_2|} s_th_{\tilde{{J_2}}} \wedge e_{I_1} \otimes s_t q h_{\tilde{{J_1}}} \wedge q e_{I_2})\\
&=& (-1)^{|I_1|~|J_2|} e_{\tilde{{I_2}}} \wedge s_t h_{J_1}  \otimes  q e_{\tilde{{I_1}}} \wedge  s_t q h_{J_2} \\
&=& (-1)^{|I_2|~|J_1|}  s_t h_{J_1} \wedge e_{I_2}  \otimes  s_t q h_{J_2} \wedge  q e_{I_1} \\
&=&T_{I_2I_1 }  ( h_{J_1} \otimes q h_{J_2}). 
\eeast
In the above equation we have used the property that two vector anti commute under wedge product. 
The above equation says that $\CJ T_{I_1 I_2 } \CJ = T_{I_2I_1 } \in \alpha_t(M)^\prime S_t$. That is
$T_{I_1 I_2 } \in \CJ \alpha_t(M)^\prime S_t \CJ = \CJ \alpha_t(M)^\prime \CJ S_t $, since $ \CJ$ commute with $S_t$.
So from the theorem \ref{bimod} it is clear that  $T_{I_1I_2} \in H_t$, i.e. 
\[
[ \{ T_{I_1 I_2 } : I_1, I_2 \in F(\CP) , (-1)^{|I_1|} = (-1)^{|I_2|} \}] \subset H_t.  
\]
To prove the equality we show that if $ T \in H_t $ and 
$T \in [ \{ T_{I_1 I_2 } : I_1, I_2 \in F(\CP) , (-1)^{|I_1|} = (-1)^{|I_2|} \}]^{\perp}$,  then $T = 0 $.
$T \in H_t $ implies that  $\alpha_t (x)T = Tx $ for all $ x \in M_R $, i.e.
$\alpha_t (x)T\Omega \otimes \Omega  = Tx\Omega \otimes \Omega $. As $\Omega \otimes \Omega$ is cyclic 
 for $M_R $, conclude that $ T$ is determined by its action on $\Omega \otimes \Omega$. So 
to prove $T=0$, it is enough to prove that  $T ( \Omega \otimes \Omega ) = 0 $. If $ I_1, I_2 \in F(\CF)$ with
$(-1)^{|I_1|} =(-1)^{|I_2|}$, then we notice that 
\beast 
\langle T ( \Omega \otimes \Omega ) , e_{I_1} \otimes qe_{I_2} \rangle &=& \langle T ( \Omega \otimes \Omega ) , T_{I_1I_2} (\Omega \otimes \Omega) \rangle\\
&=& \langle T^*_{I_1I_2} T ( \Omega \otimes \Omega ) , T_{I_1I_2} (\Omega \otimes \Omega)\rangle\\ 
&=& \langle  T_{I_1I_2},  T \rangle \langle  ( \Omega \otimes \Omega ) , T_{I_1I_2} (\Omega \otimes \Omega)\rangle\\ 
&=& 0 ~\text{ since } \langle  T_{I_1I_2},  T \rangle = 0 . 
\eeast
So we have $T( \Omega \otimes \Omega ) \in [\{ v_{I_1} \otimes qv_{I_2} :  I_1, I_2 \in F(\CP),  (-1)^{|I_1|} = (-1)^{|I_2|} \} ]^\perp$,
but Theorem \ref{prdtcsy} says that $T( \Omega \otimes \Omega ) 
\in [\{ v_{I_1} \otimes qv_{I_2} :  I_1, I_2 \in F(\CP),  (-1)^{|I_1|} = (-1)^{|I_2|} \} ]$. So $T \Omega \otimes \Omega $ has 
to be zero. So we get 
\[
H_t = [ \{ T_{I_1 I_2 } : I_1, I_2 \in F(\CP) , (-1)^{|I_1|} = (-1)^{|I_2|} \}].
\]
\end{proof}
Now write $ H = \{ ( t, x_t ) : t >0,~  x_t \in H_t \}$. Obviously $H$ is the 
super product system for {\bf CAR flow}.

\section{CCR and CAR flow }
We have already described CAR flow on type $III$ factors. 
Let us describe CCR flow as follows.
 
Let $\CH = L^2(0, \infty ) \otimes \CK$, where $\CK$ is a Hilbert space.
For $ n= 0, 1, 2, \cdots $ we will 
write $\CH^n$ for the symmetric tensor product of $n$ copies of 
$\CH$ for $n\geq 1$ with $\CH^0 = \C $. The symmetric Fock space 
over $H$ is defined as the direct sum of Hilbert spaces 
\[
 \CF_{+}(\CH) = \sum_{n = 0 }^ \infty {\CH}^n 
\]
The exponential map $\exp : \CH \rightarrow \CF_{+}(\CH)$ is defined by 
\[
 \exp(f) = \sum_{n=0}^\infty \frac{1} {\sqrt{n!}} f^{\otimes n }
\]
The symmetric Fock space $\CF_{+}(\CH) $ is span closure of the vectors 
of the form $\exp(f)$ , $f\in \CH$.
Now for every vector $f \in \CH $   there is a 
unique unitary operator $W(f)$ on $\CF_{+}(\CH)$
satisfies 
\[
 W(\xi) \exp (\eta) = e^{ -1/2 ||f|| - \langle g , f \rangle } \exp ( g + f ) 
\]
Let {\em CCR}$(\CH)$ be the unital $C^*$-algebra generated by $\{ W(f) : f \in \CH \}$ in 
$\CB(\CF_{+}(\CH))$. Let $T$ be a positive operator on $\CB( \CH)$. Then 
the operator $T$ determines the a state on 
{\em CCR}$(\CH)$ which satisfies the conditions; 
\[
 \varphi_T(W(f) ) = e^{ - 1/2 || (1+ 2T)^{1/2} f ||}.
\]
This is called the quasi-free state with symbol $T$. 

Consider the Hilbert space $\CH_T = \CF_{+}(\CH) \otimes \CF_{+}(\CH) $.  There exists a representation
$\pi_T$ of the $C^*$-algebra  {\em CCR}$(\CH)$ on the Hilbert space 
$\CH_T = \CF_{+}(\CH) \otimes \CF_{+}(\CH) $,
 defined by the formula
\begin{align*}
 &\pi_T (W(f)) = W((1+T)^{1/2}f) \otimes W(qT^{1/2}f),
\end{align*}
 where $ f \in \CH$              
and $q$ is an anti-unitary operator on $\CH$ with $q^2=1$ (see  \cite{BrRo}, or 
\cite{ArWo}).
In this representation,
the state  $\varphi_T$  becomes the vector state 
\[
\varphi_T (x) = \langle \Omega \otimes \Omega, \pi_T(x)\Omega \otimes \Omega\rangle,
\]
for $x \in ${\em CCR}$(\CH)$, and  $\CH_T =  \CF_+(\CH) \otimes \CF_+(\CH) = 
\overline{\pi_T(\text{{\em CCR}}(\CH)) \Omega \otimes \Omega} $ 
is the GNS Hilbert space, under the assumption that $T$ is injective 
(and hence also has dense range). So $( \pi_T , \CH_T , \Omega \otimes \Omega )$ is 
the GNS triple for the $C^*$-algebra {\em CCR}$(\CH)$ with respect to the state $\varphi_T $.
We write $M_T = \{ \pi_T (\text{{\em CCR}}(\CH)\}^{\prime \prime }$. 

\bigskip 
Let $\{s_t\}_{t\geq}$ be the shift semigroup on $\CH$ and suppose that $T$ commutes 
with $s_t$ for all $t\geq0$. Then $M_T$ is a type III factor (see \cite{Hlv}) and the CCR flow 
\cite{Arv} restricts to an $E_0$-semigroup on $M_T$, 
$ \alpha = \{ \alpha_t : t\geq 0 \} $ uniquely 
determined by the following condition:
\[
 \alpha_t( \pi_T(W(f)) = \pi_T(W(s_tf)),
\]
for all $ f \in \CH, t\geq 0 $. This $E_0$-semigroup is called {\bf CCR flow of 
 rank dim $\CK$ }.

Note that if 
$T=\frac{\lambda}{1-\lambda}$ with $\lambda\in (0,1)$, then
it is well-known that $M_\lambda=M_T$ is a type
$III_\lambda$ factor. 
Further, It has been mentioned in the section of examples of \cite{BISS} together with \cite{MS}
that  $\{\alpha_t ;t\geq 0\}$
is equi-modular and all these $E_0-$semigroups on type $III_\lambda$ 
factors are extendable.

\begin{Remark}
Type $III$ factors arising from quasi-free representation  of CCR and CAR algebras
with respect to the quasi-free states
will always be hyperfinite factors (see \cite{ArWo}). In particular in both the cases
we find hyperfinite $III_{\lambda}$ factors for $\lambda \in (0,1) \setminus \{{\half}\}$.
Since $III_{\lambda}$ factors are unique for every $\lambda \in (0,1) \setminus \{{\half}\}$,
so we have two families of $E_0$-semigroups namely CAR flows and CCR flows on the same factor. 
\end{Remark}
Now we have the following Corollary to the Theorem \ref{ntext} regarding 
the cocycle conjugacy of CAR flows and CCR flows.

\begin{Corollary}
 The CAR and CCR flows arising from quasi-free states are not cocycle conjugate.
\end{Corollary}

\begin{proof}
It has been proved in \cite{BISS} that CCR flows arising 
from these quasi-free states are extendable.
By Theorem \ref{ntext}, CAR flows arising from quasi-free states
are not extendable. But extendability of $E_0$-semigroup is a cocycle conjugacy invariant,
so the result follows.
\end{proof}

\begin{Remark}
 This result is surprising, since on the type $I$
factor CCR and CAR flows of same Arveson index are cocycle conjugate(\cite{Arv}). 
\end{Remark}

\bigskip \noindent

\begin{acknowledgements}
 
I would like to thank Professors V.S. Sunder and 
R. Srinivasan for very useful  discussions while this work 
was progressing. I will also thank Oliver T. Margetts
for bringing my attention to the comparison of CCR and CAR flows.   

\end{acknowledgements}

\bigskip \noindent

\textit{Email id. panchugopal@imsc.res.in, pg.math@gmail.com \\
Institute of Mathematical Sciences, Chennai}

\bibliographystyle{amsalpha}
\bibliography{references}

\end{document}